\documentclass{amsart}[14pt]

\usepackage{bm}
\usepackage{amsmath, amstext, amssymb, amsopn, amsthm, amsgen, amscd, 
 fouriernc, upgreek}
\usepackage{stmaryrd}
\usepackage[small,nohug,heads=vee]{diagrams} \diagramstyle[labelstyle=\scriptstyle]

\usepackage{color}
\usepackage{tikz}

\usepackage{graphicx}
\usepackage{epstopdf}

\definecolor{CadetBlue}{cmyk}{0.62, 0.57, 0.23, 0 }
\definecolor{black}{cmyk}{1, 0.5, 0, 0 }
\definecolor{RedViolet}{cmyk}{0.07, 0.9, 0, 0.34 }
\definecolor{SeaGreen}{cmyk}{0.69, 0, 0.5, 0}

\DeclareMathAlphabet{\mathpzc}{OT1}{pzc}{m}{it}

\newcommand{\R}{\mathbb R}

\newcommand{\N}{\mathbb N}
\newcommand{\Q}{\mathbb Q}

\newcommand{\T}{\mathbb T}

\newtheorem{theo}{Theorem}
\newtheorem{lemm}{Lemma}

\newtheorem{coro}{Corollary}

\theoremstyle{definition}

\theoremstyle{remark}

\newtheorem{note}{Note}

\title{Modular Invariant of Quantum Tori II: The Golden Mean}
\author{C. Casta\~{n}o Bernard \& T. M. Gendron}
\address{Centro de Estudios en F\'{i}sica y Matem\'{a}ticas B\'{a}sicas y Aplicadas, Universidad Aut\'{o}noma de Chiapas,
4a. Oriente Norte No. 1428, Colonia Barrio la Pimienta, Tuxtla Guti\'{e}rrez, Chiapas, M\'{e}xico }
\email{ccastanobernard@gmail.com}
\address{Instituto de Matem\'{a}ticas -- Unidad Cuernavaca, Universidad
Nacional Autonoma de M\'{e}xico, Av. Universidad S/N, C.P. 62210
Cuernavaca, Morelos, M\'{e}xico}
\email{tim@matcuer.unam.mx}
\date{31 January 2012}
\subjclass[2000]{11U10, 11F03, 14H52, 57R30, 11M55}
\keywords{}
\begin{document}
\vspace{2cm} \maketitle

\begin{abstract} In \cite{CBG} a modular invariant $j^{\rm qt}(\uptheta )$ of the quantum torus $\T(\uptheta )$ was defined.  In this paper,
we consider the case of $\uptheta=\upvarphi$ = the golden
mean.  We show that $j^{\rm qt}(\upvarphi )\approx 9538.249655644 $ by producing an explicit formula for it involving weighted
versions of the Rogers-Ramanujan functions.   
\end{abstract}
\tableofcontents

\section*{Introduction}

 In \cite{CBG}, using nonstandard methods, the following definition of modular invariant $j^{\rm qt}(\uptheta )$ of the quantum torus $\T(\uptheta )=\R/\langle 1,\uptheta\rangle$ was presented.  Let $\|\cdot \| =$ the distance-to-the-nearest-integer function and for each
 $\upvarepsilon >0$ let
 \[ B_{\upvarepsilon}(\uptheta ) = \big\{   n\in \N \big|\; \|n\uptheta\|<\upvarepsilon   \big\} .\]
Define
 \begin{align}\label{stdformulaepsilon} j_{\upvarepsilon}(\uptheta ) := \frac{12^{3}}{ 1-J_{\upvarepsilon}(\uptheta ) } ,& \quad J_{\upvarepsilon}(\uptheta ) :=  \frac{49}{40} \frac{\big( \sum_{n\in B_{\upvarepsilon}(\uptheta )}  n^{-6} \big)^{2}   }{\big( \sum_{n\in B_{\upvarepsilon}(\uptheta )}  n^{-4} \big)^{3}} 
 \end{align}
 and
 \[  j^{\rm qt}(\uptheta ) := \lim_{\upvarepsilon\rightarrow 0}  j_{\upvarepsilon}(\uptheta ) \]
 provided the limit exists; if not, we define $ j^{\rm qt}(\uptheta )=\infty$.    In \cite{CBG} it was shown that $ j^{\rm qt}(\uptheta )=\infty$ for all $\uptheta\in\Q$.

In this paper we study the case of $\uptheta = \upvarphi$ = the golden mean. We will show that $j^{\rm qt}(\upvarphi )\approx 9538.249655644 = 12^{3} \times 5.5198204025717$ by providing an {\it explicit formula} for $j^{\rm qt}(\upvarphi )$,
in which
\[ J^{\rm qt}(\upvarphi )=\frac{49}{40}\frac{ \big(G_{6}(\upvarphi)  + H_{6}(\upvarphi)\big)^{2}}{ \big(G_{4}(\upvarphi)  + H_{4}(\upvarphi) \big)^{3}}
\]
and where $G_{M}(x), H_{M}(x)$ are weighted variants of the Rogers-Ramanujan functions.  

For example, the function $G_{M}(x)$
is defined as follows: for each partition $I=i_{1}\leq i_{2}\leq \cdots \leq i_{k}$ of $n$, consider the weighting polynomial
$f_{I}(x) = x^{i_{1}}+\cdots + x^{i_{k}} $.
 Let $P(n)$ be the set of partitions of $n$ whose parts are distinct and whose differences
 are at least 2 and write
$C_{x,M}(n) :=x^{Mn}  \sum _{I\in P(n)}
  f_{I}(x )^{-M}$.
  Then
  \[  G_{M}(x)=\sum C_{x,M}(n) x^{n}.\]
  If one replaces $f_{I}(x)$ by the equiweight $x^{n}$ for all
    $I\in P(n)$, one recovers the variable part of the function appearing on the left hand side of the first Rogers-Ramanujan identity.

The classical Rogers-Ramanujan functions appear in Baxter's solution \cite{Bax} to the hard
hexagon model of statistical mechanics; in view of the quantum statistical mechanical
treatment of Complex Multiplication produced in \cite{CMR1}, \cite{CMR2}, it would seem not unreasonable
to ask that the weighted Rogers-Ramunujan functions or $j^{\rm qt}(\upvarphi)$ appear as partition
function or internal energy of some quantum statistical mechanical system.

\vspace{5mm} 

\noindent {\bf {\small Acknowledgments:}}  
This work was supported in part by the grants CONACyT 058537 and PAPIIT IN103708.

\section{The Golden Mean}\label{golden}

Let
\[  \upvarphi := \frac{1+\sqrt{5}}{2}\]
be the golden mean.
In this section we recall some facts about $\upvarphi$ and its diophantine approximations, see
for example \cite{Sch}, \cite{Vo}. 

The minimal polynomial of $\upvarphi$ is $X^{2}-X-1$
and $\upvarphi$ is a unit in $\Q (\sqrt{5})$, whose inverse is $-1$ times its conjugate:
\[   \upvarphi^{-1} = -\upvarphi' =\frac{\sqrt{5}-1}{2}.   \]
The discriminant of $\upvarphi$ is $\sqrt{5}$, and the class number of $\Q (\sqrt{5})$ is one.  
The pseudo lattice $\langle 1,\upvarphi\rangle$ has endomorphism ring equal to $O_{K}$, hence
has conductor $f=1$. 

If we denote by $[a_{0},a_{1},\dots ]$ the sequence of partial quotients of a real number $\uptheta$ then for
$\uptheta=\upvarphi$, $a_{i}=1$ for all $i$.
It follows that the sequence of best approximations $(p_{m},q_{m})$ of $\upvarphi$ is given by $(F_{m+1}, F_{m})$, where
 $\{F_{m}\}=\{1,1, 2,3,5,8,\dots \}$, $m\geq 1$, denotes the Fibonacci sequence:
\[ F_{m+1}= F_{m}+ F_{m-1}, \quad m\geq 1. \]
See for example \cite{Sch}.
This means that as $m\rightarrow\infty$,
\begin{equation}\label{errorterm}  \upvarepsilon_{m} := F_{m}\upvarphi -F_{m+1}\longrightarrow 0 
\end{equation}
 and that for all $0<n<F_{m}$, 
 \[  \|n\upvarphi \| > \| F_{m}\upvarphi  \| = | \upvarepsilon_{m}|, \]
 where as before $\|x\|$ is the distance of $x$ to the nearest integer.
 
 We recall Binet's formula \cite{NZM}:
 \[  F_{m}= \frac{\upvarphi^{m}-(\upvarphi')^{m}}{\sqrt{5}} = \frac{\upvarphi^{m}-(-1)^{m}\upvarphi^{-m}}{\sqrt{5}} 
 =\left\{\begin{array}{cc}
          \frac{\upvarphi^{m}-\upvarphi^{-m}}{\sqrt{5}}  & \text{if $m$ is even} \\
            \frac{\upvarphi^{m}+\upvarphi^{-m}}{\sqrt{5}}  & \text{if $m$ is odd.}
            \end{array}\right.
                        \]
 Using Binet's formula, we may obtain the following explicit expression for
 $\upvarepsilon_{m}$ of (\ref{errorterm}):
 \begin{align}\label{explicitformoferror} \upvarepsilon_{m}=(-1)^{m+1}\upvarphi^{-m}.
 \end{align}
 Indeed, for each integer $m$ we have
 \begin{align*}F_m\upvarphi - F_{m + 1}  & = 
 \left(\frac{\upvarphi^m + (-1)^{m + 1}\upvarphi^{-m}}{\sqrt{5}}\right)\upvarphi - \left(\frac{\upvarphi^{m + 1} + (-1)^{m}\upvarphi^{-m - 1}}{\sqrt{5}} \right)\\
 &  = 
 \frac{1}{\sqrt{5}}\left(\upvarphi^{m + 1} + (-1)^{m + 1}\upvarphi^{-m +  1} - \upvarphi^{m + 1} + (-1)^{m+1}\upvarphi^{-m -  1}\right)  \\
&  = 
 \frac{1}{\sqrt{5}}(\upvarphi + \upvarphi^{-1})(-1)^{m + 1}\upvarphi^{-m} = (-1)^{m + 1}\upvarphi^{-m}.
\end{align*}
Notice then that for $m\geq 2$, we have
\[  \|F_{m}\upvarphi\| = |\upvarepsilon_{m}|. \]
 For $m$ large, $\sqrt{5}F_{m}\approx\upvarphi^{m}$, with an error term = $\pm\upvarphi^{-m}$ that decays exponentially
 as $m\rightarrow\infty$.  

Finally, we recall Zeckendorf's representation (which is actually a special case of a more general
result of Ostrowski \cite{Os}):
\begin{theo}[Zeckendorf, \cite{Ze}]\label{Zeck}  Every natural number $n\in\N$ may be written uniquely as a sum of non-consecutive
Fibonacci numbers:
\[  n = F_{I} := F_{i_{1}}+\cdots + F_{i_{k}},\quad 2\leq i_{1}, i_{1}+2\leq i_{2},\dots , i_{k-1}+2\leq i_{k},\;\; 1\leq k.    \]
\end{theo}

\begin{note}  The condition that $i_{1}\geq 2$ is to ensure uniqueness in the decomposition, otherwise the value $1$ could occur in two different ways, as $F_{1}$ or $F_{2}$.
\end{note}

\section{An Explicit Formula}

In this section we will produce, assuming that $j^{\rm qt}(\upvarphi )$ converges, an explicit formula for $\upvarphi $ obtained by evaluating at $\upvarphi $ a certain rational
expression involving weighted variants of the Rogers-Ramanujan functions.  The convergence of $j^{\rm qt}(\upvarphi )$ 
will be then proved in \S \ref{golden2}.

Recall the standard formula (\ref{stdformulaepsilon}).   Write $\upvarepsilon=|\upvarepsilon_{m}|$ and
$ B=B_{m}(\upvarphi )=\big\{ n\in\N\, |\; \| n\upvarphi \|< \upvarepsilon \big\}  $
so that
\[ J^{\rm qt}_{\upvarepsilon}(\upvarphi ) =  \frac{49}{40} \frac{\big( \sum_{n\in B}  n^{-6} \big)^{2}   }{\big( \sum_{n\in B}  n^{-4} \big)^{3}}. \]
The first step is to determine the elements of $B$ in terms of their Zeckendorf representations.
In what follows, for a multi-index $I=(i_{1},\dots, i_{k})$, define $|I|=k$.

\begin{lemm}\label{ZeckLemm}  Let $n =F_{I}=F_{i_{1}}+\cdots + F_{i_{k}}$ written in its unique Zeckendorf form.  Then
$n\in B$ if and only 
\begin{itemize}
\item[I.]  $|I|\geq 1$, $i_{1}\geq m+1$ or
\item[II.] $|I|\geq 2$, $i_{1}=m$ and $ i_{2}-m$ is odd.
\end{itemize}
\end{lemm}

\begin{note} Since the Zeckendorf form consists of sums of nonconsecutive 
Fibonacci numbers, we must have that  $ i_{2}-m\geq 3$
in II. 
\end{note}

\begin{proof}  First note that we have trivially by (\ref{explicitformoferror}) that $F_{m+i}\in B$ for $i\geq 1$. Suppose that $n=F_{I}$ is a sum of more than one non-consecutive Fibonacci numbers
and $i_{1}\geq m+1$.  Then we have
\[     \|n\upvarphi\| < \upvarphi^{-(m+1)} + \upvarphi^{-(m+3)}+\cdots = \upvarphi^{-(m+1)}(1-\upvarphi^{-2})^{-1}. \]
Since $\upvarphi =\upvarphi^{2}-1$ it follows that $(1-\upvarphi^{-2})^{-1}=\upvarphi$.   Then $ \|n\upvarphi\|<\upvarphi^{-m}$ which
implies that $n\in B$.  Thus every element of the type described in I. belongs to $B$.
On the other hand, if $i_{1}\leq m-1$, then we claim that 
\[ \upvarphi^{-m} =\upvarepsilon< \|n\upvarphi\|< 1-\upvarphi^{-1} = \upvarphi^{-2}.\] 
Indeed, if $n=F_{I}$, the associated error term sum 
\[ \upvarepsilon_{I} := \pm \upvarepsilon_{i_{1}} \pm \cdots \pm \upvarepsilon_{i_{k}}\]
is minimized in absolute value by taking $i_{1}=m-1$ and assuming
that the remaining indices $i_{2},\dots$ are such that the signs of the associated error terms $\upvarepsilon_{i_{2}},\dots $ are different from the sign
of the error term $\upvarepsilon_{m-1}$.  More precisely,
\[   |\upvarepsilon_{I}| > \upvarphi^{-(m-1)} - (\upvarphi^{-(m+2)} + \upvarphi^{-(m+4)} +\cdots ) = 
\upvarphi^{-m}(\upvarphi -\upvarphi^{-2}(1-\upvarphi^{-2})^{-1}) . \]
Since $\upvarphi^{-2}(1-\upvarphi^{-2})^{-1}=\upvarphi^{-1}$ and $\upvarphi-\upvarphi^{-1}=1$, it follows that
$ |\upvarepsilon_{I}|>\upvarphi^{-m} =\upvarepsilon$.
In addition $|\upvarepsilon_{I}|$ is maximized by taking $i_{1}=2$, $i_{2}=4,\dots$, so that
\[ |\upvarepsilon_{I}| < \upvarphi^{-2} + \upvarphi^{-4}+\cdots = \frac{1}{\upvarphi^{2}-1} = \upvarphi^{-1}.  \]
Note that the distance of the latter bound $\upvarphi^{-1}$ to the nearest integer is $1-\upvarphi^{-1}=\upvarphi^{-2}$.
It follows then from the definition of $\| \cdot \|$ and the fact that we are assuming that $m>2$ that
$ \|n\upvarphi\|>\upvarphi^{-m} =\upvarepsilon$ and
$n\not\in B$.
Now if $i_{1}=m$ and $i_{2}-m$ is even, then the error terms $\upvarepsilon_{m}$ and
$\upvarepsilon_{i_{2}}$ share the same sign, and we have 
\[   \|n\upvarphi\|> \upvarphi^{-m} + \upvarphi^{-i_{2}} - \left(\upvarphi^{-(i_{2} +3)} + \upvarphi^{-(i_{2} +5)} + \cdots \right) 
= \upvarepsilon + (\upvarphi^{-i_{2}} - \upvarphi^{-(i_{2} +3)}(1-\upvarphi^{-2})^{-1}) >\upvarepsilon 
\]
Indeed, the last inequality follows since 
\begin{align*} 
 & \upvarphi^{-i_{2}} - \upvarphi^{-(i_{2} + 3)}(1-\upvarphi^{-2})^{-1} =  \\
 & \upvarphi^{-i_{2}}(1 - \upvarphi^{-3}(1 - \upvarphi^{-2})^{-1}) = \\
 & \upvarphi^{-i_{2}}(1 - \upvarphi^{-2}(\upvarphi - \upvarphi^{-1})^{-1}) = \\
 & \upvarphi^{-i_{2}}(1-\upvarphi^{-2}) >0.
\end{align*}
On the other hand, if $i_{1}=m$ and $i_{2}=m+k$, $k$ odd, then
the sign of the corresponding error terms differ, and we have
\begin{align*}  \|n\upvarphi\| & < \upvarphi^{-m}-\upvarphi^{-m-k}+\upvarphi^{-m-k-3}+\upvarphi^{-m-k-5}+\cdots  \\
 & = \upvarphi^{-m} - \upvarphi^{-m-k}\left( 1- (\upvarphi^{-3}+\upvarphi^{-5}+\cdots )   \right) \\
 & =  \upvarphi^{-m} - \upvarphi^{-m-k}\left( 1-\upvarphi^{-3}(1-\upvarphi^{-2})^{-1}\right) \\
& =  \upvarphi^{-m} - \upvarphi^{-m-k}\left( 1-\upvarphi^{-2}\right) <\upvarepsilon 
\end{align*}
so that $n\in B$.
\end{proof}

    Let $\mathfrak{N} $ be the set of increasing, finite tuples $I=(i_{1},\dots , i_{l})$ of natural numbers with $|I|=l\geq 2$ and
which are not consecutive i.e. $ i_{1}+2\leq i_{2}, \dots , i_{l-1}+2\leq i_{l}$.  Denote
by 
\begin{align}\label{defNm}
 \mathfrak{N}(m) & = \{ I=(i_{1},\dots , i_{l})\in \mathfrak{N}|\; i_{1}\geq m\}.
 \end{align} 
Also denote by 
\begin{align}\label{defMm} 
\mathfrak{M}(m) & = \{ I\in \mathfrak{N}(m)|\; i_{1}=m \text{ and } i_{2}=m+k \text{ for }
k \text{ odd}\}.
\end{align}

Consider $B_{m}$ for $m>2$.  Then by the Lemma we have  
\[     J^{\rm qt}_{B_{m}} (\upvarphi ) : = \frac{49}{40}\frac{ \bigg(  \sum_{n\in B_{m}} n^{-6}  \bigg)^{2}  }{\bigg(\sum_{n\in B_{m}} n^{-4}  \bigg)^{3}} = \frac{49}{40}\frac{  \bigg( \sum_{i=1}^{\infty}  F_{m+i}^{-6}  +   \sum_{I\in\mathfrak{N}(m+1)}  F_{I}^{-6}  + 
 \sum_{ I\in \mathfrak{M}(m)} F_{I}^{-6} \bigg)^{2} }
 {  \bigg( \sum_{i=1}^{\infty}  F_{m+i}^{-4}  +   \sum_{I\in\mathfrak{N}(m+1)}  F_{I}^{-4}  + 
 \sum_{ I\in \mathfrak{M}(m)} F_{I}^{-4} \bigg)^{3}}  , \]
 an expression whose status is still only formal.  Consider also the formal expression
 \begin{align}\label{limitoftheJs}
  J^{\rm qt}(\upvarphi ) & := \frac{49}{40}\frac{ \bigg( \sum_{i= 1}^{\infty}  \upvarphi^{-6i}  +   \sum_{I\in\mathfrak{N}(1)}   \upvarphi_{I}  ^{-6}  + 
 \sum_{ I\in \mathfrak{M}(0)} \upvarphi_{I}^{-6} \bigg)^{2} }
 { \bigg( \sum_{i=1}^{\infty} \upvarphi^{-4i}   +   \sum_{I\in\mathfrak{N}(1)}   \upvarphi_{I}  ^{-4}  + 
 \sum_{ I\in \mathfrak{M}(0)} \upvarphi_{I}^{-4} \bigg)^{3}}    
 \end{align}
 where 
 \[   \upvarphi_{I} := \upvarphi^{i_{1}} +\cdots +  \upvarphi^{i_{l}} .\]

\begin{theo}\label{limitgoldenmean} If $J^{\rm qt}(\upvarphi )$ converges then so does $J^{\rm qt}_{B_{m}}(\upvarphi )$ for each $m$ and
 \[    J^{\rm qt}_{B_{m}}(\upvarphi ) \longrightarrow J^{\rm qt}(\upvarphi )\]
 as $m\rightarrow\infty$.
 \end{theo}

 \begin{proof}
 Multiply the numerator and denominator of $J^{\rm qt}_{B_{m}}(\upvarphi )$ by $F^{12}_{m}$
to obtain
\begin{align}\label{FibNormOfJ} J^{\rm qt}_{B_{m}}(\upvarphi ) & = \frac{49}{40}\frac{  \bigg( \sum_{i=1}^{\infty}  (F_{m}/F_{m+i})^{6}  +   \sum_{I\in\mathfrak{N}(m+1)}  (F_{m}/F_{I})^{6}  + 
 \sum_{ I\in \mathfrak{M}(m)} (F_{m}/F_{I})^{6} \bigg)^{2} }
 {  \bigg( \sum_{i=1}^{\infty}  (F_{m}/F_{m+i})^{4}  +   \sum_{I\in\mathfrak{N}(m+1)}  (F_{m}/F_{I})^{4}  + 
 \sum_{ I\in \mathfrak{M}(m)} (F_{m}/F_{I})^{4}\bigg)^{3}}  .  
 \end{align}

It will suffice to show that each term $T^{-6}_{m}=T^{-6}_{m,I}$  ($T^{-4}_{m}=T^{-4}_{m,I}$) appearing in a sum contained in the numerator (denominator) of (\ref{FibNormOfJ})
satisfies 
\[ C_{m}^{-6} \cdot T^{-6} < T^{-6}_{m} <C_{m}^{6} \cdot T^{-6}  \quad \bigg(C_{m}^{-4} \cdot T^{-4} < T^{-4}_{m} <C_{m}^{4} \cdot T^{-4}  \bigg)   \]
where $T=T_{I}$ is the correspondingly indexed term of $J^{\rm qt}(\upvarphi )$ and
\[  C_{m} = \frac{1+\upvarphi^{-2m}}{1-\upvarphi^{-2m}}.\]
This will give convergence of each $J^{\rm qt}_{B_{m}}(\upvarphi )$, as well as the bound
\[   \left( \frac{1-\upvarphi^{-2m}}{1+\upvarphi^{-2m}}\right)^{24}J^{\rm qt}(\upvarphi )<  J^{\rm qt}_{B_{m}}(\upvarphi ) < \left( \frac{1+\upvarphi^{-2m}}{1-\upvarphi^{-2m}}\right)^{24}J^{\rm qt}(\upvarphi ) ,  \]
which implies that $J^{\rm qt}_{B_{m}}(\upvarphi ) \rightarrow J^{\rm qt}(\upvarphi )$.

We will now make use of Binet's formula, $\sqrt{5}F_{m}=(\upvarphi^{m}\pm\upvarphi^{-m})$.  
Note that
the $\sqrt{5}$ factors drop out and so we may simply replace every Fibonacci term $F_{m}$ appearing by
$\upvarphi^{m}\pm\upvarphi^{-m}$.

We consider first the numerator of (\ref{FibNormOfJ}), treating each of the three sums there separately.  The
first sum may be written \[   \sum_{i=1}^{\infty}  (F_{m}/F_{m+i})^{6} = \sum_{i=1}^{\infty}  \left( \frac{\upvarphi^{m} \pm\upvarphi^{-m}}{ \upvarphi^{m+i} \pm(-1)^{i} \upvarphi^{-(m+i)}}  \right)^{6} 
 =\sum_{i=1}^{\infty} \upvarphi^{-6i}  \left( \frac{1 \pm\upvarphi^{-2m}}{1 \pm(-1)^{i}  \upvarphi^{-2m-2i}}  \right)^{6} .        \]
Note that
 \[  \left( \frac{1-\upvarphi^{-2m}}{1+\upvarphi^{-2m}}\right)^{6} <   \left( \frac{1 \pm\upvarphi^{-2m}}{1 \pm(-1)^{i}  \upvarphi^{-2m-2i}}  \right)^{6} <
\left( \frac{1+\upvarphi^{-2m}}{1-\upvarphi^{-2m}}\right)^{6}  .  \]
The next sum is 
\begin{align}\label{BinetNormOfJ}
     \sum_{I\in\mathfrak{N}(m+1)}   \left( F_{m}/ F_{I}\right)^{6} & = \sum_{I\in\mathfrak{N}(m+1)}   \left(\frac{ \upvarphi^{m}\pm\upvarphi^{-m} }{ (\upvarphi^{m+i_{1}}\pm \upvarphi^{-m-i_{1}} )+\cdots + (\upvarphi^{m+i_{k}}\pm\upvarphi^{-m-i_{k}})  } \right)^{6}, 
     \end{align}
where we are writing our generic $I\in\mathfrak{N}(m+1)$ in the form $I=(i_{1}+m,\dots ,i_{k}+m)$ with $1\leq i_{1}<i_{2}<\cdots <i_{k}$. 
Letting $I_{0}=(i_{1},\dots , i_{k})$ then each term of the sum in (\ref{BinetNormOfJ}) may be re-written
\begin{align}\label{rewritingtheterms}  \left( \frac{1\pm\upvarphi^{-2m}}{ \upvarphi_{I_{0}} + (\pm\upvarphi_{-I_{0}-2m}) }\right)^{6}
=   \upvarphi_{I_{0}}^{-6} \cdot\left( \frac{1\pm\upvarphi^{-2m}}{1 + (\pm \upvarphi_{-I_{0}-2m})/  \upvarphi_{I_{0}}  }\right)^{6}
 \end{align}
 where 
 \[\pm\upvarphi_{-I_{0}-2m} :=\pm \upvarphi^{-i_{1}-2m}\pm \cdots\pm  \upvarphi^{-i_{k}-2m}, \]
 the signs determined as in Binet's formula by the parities of the powers.
It is easy to see that
 \begin{align}\label{basicinequalities}
   \left( \frac{1-\upvarphi^{-2m}}{1+\upvarphi^{-2m}}\right)^{6}
   < \left( \frac{1\pm\upvarphi^{-2m}}{1 + (\pm \upvarphi_{-I_{0}-2m})/  \upvarphi_{I_{0}}  }\right)^{6}< \left( \frac{1+\upvarphi^{-2m}}{1-\upvarphi^{-2m}}\right)^{6}:
   \end{align}
indeed, both inequalities in (\ref{basicinequalities}) follow since
 \[  \upvarphi^{-2m} > (\pm \upvarphi_{-I_{0}-2m})/  \upvarphi_{I_{0}} >-\upvarphi^{-2m},\]
 true as
 \begin{align}\label{trueas}  (\pm \upvarphi_{-I_{0}-2m})/\upvarphi_{I_{0}} = \upvarphi^{-2m}
\left( \frac{\pm\upvarphi^{-i_{1}}\pm \cdots\pm  \upvarphi^{-i_{k}} }{\upvarphi^{i_{1}}+ \cdots+  \upvarphi^{i_{k}}}\right).
  \end{align}
 What remains is the sum over $\mathfrak{M}(m)$: the analysis here is essentially the same as that made
 for the sum over $\mathfrak{N}(m+1)$, only we take into account that $I=(m,m+j,m+i_{3},\dots ,m+i_{k})$
 where $j$ is odd.  Writing $I_{0}=(0,j,i_{3},\dots ,i_{k})$, then we have the equation (\ref{rewritingtheterms}) with
 \[  \pm\upvarphi_{-I_{0}-2m} =\pm\upvarphi^{-2m}\mp\upvarphi^{-j-2m} \pm\cdots\pm  \upvarphi^{-i_{k}-2m}, \]
where the $\mp$ sign of $\upvarphi^{-j-2m}$ indicates that this sign is opposite to that of $\upvarphi^{-2m}$, as  
 $j$ is odd.  The analogue of (\ref{trueas}) is then
 \begin{align*} (\pm \upvarphi_{-I_{0}-2m})/\upvarphi_{I_{0}} = \upvarphi^{-2m}
\left( \frac{\pm1\mp\upvarphi^{-j}\pm\cdots\pm  \upvarphi^{-i_{k}} }{1 +\upvarphi^{j}+ \cdots+  \upvarphi^{i_{k}}}\right),
  \end{align*}
  which yields the analogue of (\ref{basicinequalities}) in this case.  This completes our bounding of
  the numerator.  Analogous bounds, with the exponent $6$ replaced by $4$, may be found for the corresponding sums in the denominator of $J^{\rm qt}_{B_{m}}$.
 The result now follows.
 \end{proof}
 
Let $P(n)$ be the set of partitions of $n$ into
  into distinct parts whose differences are at least $2$, and let $c(n)=|P(n)|$. 
   The generating function
  \[  F(x)=\sum c(n)x^{n}=\sum \frac{x^{n^{2}}}{(1-x)\cdots (1-x^{n})}\] 
 is of substantial combinatorial interest: it is the left-hand side of the first Rogers-Ramanujan identity \cite{HaWr}.
 
For each partition $I\in P(n)$, let $f_{I}(x)=x^{i_{1}}+\cdots +x^{i_{k}}$
be the associated weighting polynomial.
Define
 \[ C_{x,M}(n)=x^{Mn} \sum _{I\in P(n) }
  f_{I}(x)^{-M}.\]
  Considering the generating function
  \[  G_{M}(x)=\sum C_{x,M}(n) x^{n}.\]
  Clearly we have 
  \[   G_{M}(\upvarphi) =  \sum_{i= 1}^{\infty}  \upvarphi^{-Mi}  +   \sum_{I\in\mathfrak{N}(1)}   \upvarphi_{I}  ^{-M} . \]
  Similarly, let $Q(n)\subset P(n)$ 
be the set of those partitions $I=i_{1}<i_{2}<\cdots <i_{k}$ in $P(n)$ for which $i_{1}$ is odd and $\geq 3$.
Let 
\[ D_{x,M}(n) :=x^{Mn}  \sum _{I\in Q(n)}(1+f_{I}(x ))^{-M}\]
and define \[  H_{M}(x):=\sum D_{x,M}(n) x^{n}.\]
Then 
 \[   H_{M}(\upvarphi) =  \sum_{ I\in \mathfrak{M}(0)} \upvarphi_{I}^{-M}.\]
 The following is then immediate:

\begin{coro}\label{explicitformula}  Let $J^{\rm qt}(\upvarphi )$ be as above.  Then
\[ J^{\rm qt}(\upvarphi )=\frac{49}{40}\frac{ \big(G_{6}(\upvarphi)  + H_{6}(\upvarphi)\big)^{2}}{ \big(G_{4}(\upvarphi)  + H_{4}(\upvarphi) \big)^{3}} .
\] 
\end{coro}

 \begin{note} 
  If one replaces in the formula for $C_{x, M}(n)$ the weighting polynomial $ f_{I}(x )^{-M}$ by the equiweight $x^{-Mn}$ one recovers $c(n)$.  Thus the functions $G_{M}(x), H_{M}(x)$ may be viewed as weighted variants of the variable part of the Rogers-Ramanujan function.
  \end{note}

 \section{Convergence}\label{golden2}
 
In this section we will show that $j(\upvarphi )^{\rm qt}<\infty$.  As before we write
 $   j^{\rm qt}(\upvarphi ) := 12^{3}/( 1-J^{\rm qt}(\upvarphi )  ) $.  
 
\begin{theo}\label{goldenbound} $j(\upvarphi )^{\rm qt}$ converges with the bounds 
 \[   9150< j^{\rm qt}(\upvarphi ) <9840.  \]
 \end{theo}

 \begin{proof}  To prove the convergence of
 $j^{\rm qt}(\upvarphi )$, it is enough to prove convergence of the explicit formula $j^{\rm qt}(\upvarphi )$
 obtained from (\ref{limitoftheJs}).
 Observe first that
 \[ \sum_{i= 1}^{\infty}  \upvarphi^{-6i} = (\upvarphi^{6}-1)^{-1},\quad\quad  \sum_{i=1}^{\infty}  \upvarphi^{-4i} = (\upvarphi^{4}-1)^{-1}  \]
 so we may write
  \[    J^{\rm qt}(\upvarphi ) = \frac{49}{40}\frac{ \left( (\upvarphi^{6}-1)^{-1}  +   \sum_{I\in\mathfrak{N}(1)}   \upvarphi_{I}  ^{-6}  + 
 \sum_{ I\in \mathfrak{M}(0)} \upvarphi_{I}^{-6} \right)^{2} }
 { \left( (\upvarphi^{4}-1)^{-1} + \sum_{I\in\mathfrak{N}(1)}   \upvarphi_{I}  ^{-4}  + 
 \sum_{ I\in \mathfrak{M}(0)} \upvarphi_{I}^{-4} \right)^{3}}  .  \]
 We now find an explicit approximation and an upper bound for the sum  $\sum_{I\in\mathfrak{N}(1)}   \upvarphi_{I}^{-M} $
 where $M$ is a positive integer.  
 In fact, we will show that
 \begin{align}\label{goldennumerator1}
  \sum_{I\in\mathfrak{N}(1)}   \upvarphi_{I}^{-M}  = \frac{1}{(\upvarphi^{M}-1)(\upvarphi^{2}+1)^{M}} +C(M)  
  \end{align}
 where
 \begin{align}\label{goldennumeratorbound}
 C(M) < \widetilde{C}(M):=\frac{1}{\upvarphi^{2M}(\upvarphi^{M}-1)^{2}}   +\frac{1}{\upvarphi^{M}(\upvarphi^{M}-1)^{2}(\upvarphi^{2M}-\upvarphi^{M}-1)}. 
 \end{align}
 
 Consider first the sum of those $I$
 with $|I|=2$:
 \begin{align}
  \mathop{ \sum_{i_{1}\geq 1} }_ {i_{2}\geq i_{1}+2} \frac{1}{( \upvarphi^{i_{1}} + \upvarphi^{i_{2}}   )^{M}} & =
 \sum_{i=1}^{\infty} \upvarphi^{-Mi}\sum_{k=2}^{\infty} (1+ \upvarphi^{k})^{-M} \nonumber \\
  & = \frac{1}{\upvarphi^{M}-1}\left\{ \frac{1}{(1+\upvarphi^{2})^{M}} +\sum_{k=3}^{\infty} (1+ \upvarphi^{k})^{-M}  \right\} \label{goldenfirstorder}\\
& <  \frac{1}{\upvarphi^{M}-1}\left\{ \frac{1}{(1+\upvarphi^{2})^{M}}+\sum_{k=3}^{\infty} \upvarphi^{-Mk}  \right\}\nonumber \\
& = \frac{1}{(\upvarphi^{M}-1)(\upvarphi^{2}+1)^{M}} + \frac{1}{\upvarphi^{2M}(\upvarphi^{M}-1)^{2}} .\label{goldenfirstbound}
  \end{align}
  The equality (\ref{goldenfirstorder}) produces the explicit term 
  $1/((\upvarphi^{M}-1)(\upvarphi^{2}+1)^{M})$ appearing in (\ref{goldennumerator1}); the second term in (\ref{goldenfirstbound})
  is the first bounding term 
  in (\ref{goldennumeratorbound}).
  
  For $|I|=3$ we have
 \begin{align*} 
  \mathop{ \sum_{i_{1}\geq 1} }_ {i_{2}\geq i_{1}+2, i_{3}\geq i_{2}+2 } \frac{1}{( \upvarphi^{i_{1}} + \upvarphi^{i_{2}} + \upvarphi^{i_{3}} )^{M}}  &  
 =  \mathop{ \sum_{i_{1}\geq 1} }_ {i_{2}\geq i_{1}+2, i_{3}\geq i_{2}+2 } \upvarphi^{-Mi_{1}}\frac{1}{( 1 + \upvarphi^{i_{2}-i_{1}} + \upvarphi^{i_{3}-i_{1}} )^{M}} \\
 & <   \mathop{ \sum_{i_{1}\geq 1} }_ {i_{2}\geq i_{1}+2, i_{3}\geq i_{2}+2 } \upvarphi^{-Mi_{1}}\frac{1}{( \upvarphi^{i_{2}-i_{1}} + \upvarphi^{i_{3}-i_{1}} )^{M}} \\
 & = \mathop{ \sum_{i_{1}\geq 1} }_ {i_{2}\geq i_{1}+2, i_{3}\geq i_{2}+2 } \upvarphi^{-Mi_{1}}\upvarphi^{-M(i_{2}-i_{1})}\frac{1}{(1+ \upvarphi^{i_{3}-i_{2}} )^{M}} \\
& < \sum_{i\geq 1} \upvarphi^{-Mi}
 \sum_{j\geq 2}\upvarphi^{-Mj}\sum_{k\geq 2} \upvarphi^{-Mk} \\
 & = \frac{(\upvarphi^{-M})^{2}}{ (\upvarphi^{M}-1)^{3} }  .
 \end{align*}
 Inductively, for the terms with $|I|=l\geq 3$ we have the bound
 \[  \frac{(\upvarphi^{-M})^{l-1}}{ (\upvarphi^{M}-1)^{l} }.  \]
 Summing these bounds from $l=3$ to $\infty$ gives the second term in (\ref{goldennumeratorbound}):
 \[\sum_{l=3}^{\infty} \frac{(\upvarphi^{-M})^{l-1}}{ (\upvarphi^{M}-1)^{l} }
 = \upvarphi^{M}\sum_{l=3}^{\infty} \frac{1}{ (\upvarphi^{M}(\upvarphi^{M}-1))^{l} }
 =\frac{1}{\upvarphi^{M}(\upvarphi^{M}-1)^{2}(\upvarphi^{2M}-\upvarphi^{M}-1)}
 \]
 
 We now bound the second type of sum appearing in $J^{\rm qt}(\upvarphi)$, 
$ \sum_{ I\in \mathfrak{M}(0)} \upvarphi_{I}^{-M}  $.  We will show here that
\begin{align}
\sum_{ I\in \mathfrak{M}(0)} \upvarphi_{I}^{-M} = \frac{1}{(1+\upvarphi^{3})^{M}} + D(M)
\end{align}
where
\begin{align}
D(M) <\widetilde{D}(M):=\frac{1}{\upvarphi^{3M}(\upvarphi^{2M}-1)}  +   \frac{1}{\upvarphi^{M}(\upvarphi^{2M}-1)(\upvarphi^{2M}-\upvarphi^{M}-1)}
\end{align}

When $|I|=2$ we have, since $i_{1}=0$, that $i_{2}=2j+1$ is odd, where $j\geq 1$ (recall the definition
of $\mathfrak{M}(m)$ found in (\ref{defMm})).  For such $I$ we have the contribution
\begin{align}
  \mathop{ \sum_{i=2j+1} }_{j\geq 1}\frac{1}{(1+ \upvarphi^{i})^{M}} 
& = \frac{1}{(1+\upvarphi^{3})^{M}} + \sum_{j= 2}^{\infty} (1+\upvarphi^{(2j+1)} )^{-M} \\
& <  \frac{1}{(1+\upvarphi^{3})^{M}} +\sum_{j= 2}^{\infty} \upvarphi^{-M(2j+1)} \nonumber \\
& = \frac{1}{(1+\upvarphi^{3})^{M}} + \upvarphi^{-M}\sum_{j= 2}^{\infty}\upvarphi^{-2Mj} \nonumber \\
& = \frac{1}{(1+\upvarphi^{3})^{M}} + \upvarphi^{-5M}\frac{1}{1-\upvarphi^{-2M}} \nonumber \\
& = \frac{1}{(1+\upvarphi^{3})^{M}}+ \frac{1}{\upvarphi^{3M}(\upvarphi^{2M}-1)} .
 \end{align}

For $|I|=3$ we have
 \begin{align*} 
   \sum_ {j\geq 1, k\geq (2j+1)+2 } \frac{1}{( 1 + \upvarphi^{2j+1} + \upvarphi^{k} )^{M}}  &  
 < \sum_ {j\geq 1, k\geq (2j+1)+2 } \upvarphi^{-M(2j+1)}\frac{1}{( 1+ \upvarphi^{k-(2j+1)} )^{M}} \\
 &< \sum_{j=1}^{\infty}\upvarphi^{-M(2j+1)}\sum_{k=2}^{\infty}\upvarphi^{-Mk} \\
  & = \frac{1}{\upvarphi^{M}(\upvarphi^{2M}-1)}\cdot \frac{1}{\upvarphi^{M}(\upvarphi^{M}-1)} \\
  & = \frac{1}{\upvarphi^{M}+1}\cdot \left( \frac{\upvarphi^{-M}}{\upvarphi^{M}-1}\right)^{2} \\ 
 \end{align*}
For the sum over $I$ with $|I|=l$, we obtain inductively 
the bound
\[       \frac{1}{\upvarphi^{M}+1}\left( \frac{\upvarphi^{-M}}{\upvarphi^{M}-1}  \right)^{l-1}    \]
and summing these from $l=3$ to $\infty$ gives
\[   \frac{1}{\upvarphi^{M}(\upvarphi^{2M}-1)(\upvarphi^{2M}-\upvarphi^{M}-1)}. \]
It follows then that
\begin{align*}
  J^{\rm qt}(\upvarphi ) & <\frac{49}{40} \frac{\left( (\upvarphi^{6}-1)^{-1}  + \big( (\upvarphi^{6}-1)(\upvarphi^{2}+1)^{6}\big)^{-1}
+(1+\upvarphi^{3})^{-6} + \widetilde{C}(6) + \widetilde{D}(6) \right)^{2}}{ \left( (\upvarphi^{4}-1)^{-1}+\big( (\upvarphi^{4}-1)(\upvarphi^{2}+1)^{4}\big)^{-1}
+(1+\upvarphi^{3})^{-4}\right)^{3}} \\
& \\
& \approx 0.824376700276.
 \end{align*}
A lower bound may be given by 
\begin{align*}
  0.81115979990388 & \approx\frac{49}{40}\frac{\left( (\upvarphi^{6}-1)^{-1}  + \big( (\upvarphi^{6}-1)(\upvarphi^{2}+1)^{6}\big)^{-1}
+(1+\upvarphi^{3})^{-6} \right)^{2}}{ \left( (\upvarphi^{4}-1)^{-1}+\big( (\upvarphi^{4}-1)(\upvarphi^{2}+1)^{4}\big)^{-1}
+(1+\upvarphi^{3})^{-4}+ \widetilde{C}(4) + \widetilde{D}(4) \right)^{3}} \\
& \\
& <J^{\rm qt}(\upvarphi ) 
\end{align*}
which give the bounds presented in the statement of the theorem.
\end{proof}

\begin{note}  Using a program such as PARI one can calculate using the explicit formula of Corollary \ref{explicitformula}
that $j^{\rm qt}(\upvarphi )\approx 9538.249655644 $.
\end{note}

\bibliographystyle{amsalpha}

\end{document}